\documentclass[12pt,reqno,twoside]{amsart}

\usepackage{amssymb}
\usepackage{tikz}
\usepackage{amsmath}
\usepackage[latin1]{inputenc}
\usepackage[colorlinks=true,linkcolor=blue,citecolor=blue]{hyperref}

\usepackage[all]{xy}
\usepackage{pstricks, pst-3d}
\newcommand{\mm}{\mathfrak m}
\newcommand{\nn}{\mathfrak n}
\newcommand{\pp}{\mathfrak p}
\newcommand{\qq}{\mathfrak q}

\newcommand{\A}{\mathbb{A}}

\newcommand{\Z}{\mathbb{Z}}
\newcommand{\R}{\mathbb{R}}
\newcommand{\N}{\mathbb{N}}
\newcommand{\Q}{\mathbb{Q}}

\newcommand{\Mcc}{\mathcal{M}}

\DeclareMathOperator{\pnt}{\raise 0.5mm \hbox{\large\bf.}}

\DeclareMathOperator{\cx}{cx}
\DeclareMathOperator{\curv}{curv}

\DeclareMathOperator{\Span}{Span}
\DeclareMathOperator{\Spec}{Spec}
\DeclareMathOperator{\supp}{supp}

\DeclareMathOperator{\chara}{char}
\DeclareMathOperator{\Tor}{Tor}
\DeclareMathOperator{\Ker}{Ker}
\DeclareMathOperator{\projdim}{proj\,dim}

\def\+#1{\relax\ifmmode\if\noexpand #1\relax \mathop{\kern
    0pt^+{#1}}\nolimits\else \kern 0pt^+\!#1 \fi\else$^*$#1\fi}

\let\phi=\varphi
\let\:=\colon

\newtheorem{thm}{\bf Theorem}[section]
\newtheorem{lem}[thm]{\bf Lemma}

\newtheorem{cor}[thm]{\bf Corollary}
\newtheorem{prop}[thm]{\bf Proposition}
\newtheorem{conj}[thm]{\bf Conjecture}

\theoremstyle{definition}
\newtheorem{defn}[thm]{\bf Definition}

\newtheorem{rem}[thm]{\bf Remark}
\newtheorem{ex}[thm]{\bf Example}

\theoremstyle{plain}
\newtheorem*{thm*}{Theorem}
\newtheorem*{claim*}{Claim}

\title{Algebra retracts and Stanley-Reisner rings}

\author{Neil Epstein}
\address{Department of Mathematical Sciences \\ George Mason University \\ Fairfax, VA 22030}
\email{nepstei2@gmu.edu}

\author{Hop D. Nguyen}
\address{Am Herrenberge 11, Appartment 419 \\ 07745 Jena, Germany}
\address{Dipartimento di Matematica, Universit\`a di Genova, Via Dodecaneso 35, 16146 Genoa, Italy}
\email{ngdhop@gmail.com}
\date{\today}
\thanks{The second named author was supported by the graduate school ``Combinatorial structures in Algebra and Topology" at the University of Osnabr\"uck, Germany.}


\begin{document}

\begin{abstract}
In a paper from 2002, Bruns and Gubeladze conjectured that graded algebra retracts of polytopal algebras over a field $k$ are again polytopal algebras. Motivated by this conjecture, we prove that graded algebra retracts of Stanley-Reisner rings over a field $k$ are again Stanley-Reisner rings. Extending this result further, we give partial evidence for a conjecture saying that monomial quotients of standard graded polynomial rings over $k$  descend along graded algebra retracts. 
\end{abstract}

\maketitle

\section{Introduction}
\label{intro}
Let $\theta: R\to S$ be a ring extension that admits a section $\phi: S\to R$, so that $\phi\circ \theta=\text{id}_R$. We call $\theta$ an algebra retract with retraction map $\phi$. In this situation, $R$ is said to be an algebra retract of $S$. Equivalently, $R$ is an algebra retract of $S$ if there is an idempotent ring endomorphism of $S$ whose image is $R$. If $R$ and $S$ are graded rings, we assume
that the morphisms involved preserve the gradings. Examples of algebra retracts include the polynomial extensions $R\to R[x]$ or the tensor product
extensions $R\to R\otimes_k S$ where $R, S$ are affine algebras over a field $k$.

Several authors have considered algebra retracts to various ends; see \cite{Cos}, \cite{Her}, \cite{BG02}. For example, the study of Zariski's
cancellation problem is closely related to the study of algebra retracts.  Zariski's cancellation problem asks: let $R$ be a  $k$-algebra, where $k$ is a field, such that $R[x] \cong k[x_1,\ldots,x_n]$ for some $n\ge 1$.  Is it true then that $R\cong k[x_1,\ldots,x_{n-1}]$? Costa \cite{Cos} asked a
stronger question: if $R$ is an algebra retract of a polynomial ring over $k$, is it true that $R$ is isomorphic to a polynomial ring (over $k$)? He
proved that this is indeed the case for retracts of a polynomial ring in $2$ variables \cite[Thm.~3.5]{Cos}. Recently, N. Gupta \cite{Gupta}
provided a counterexample for Zariski's cancellation problem in $3$ variables.

A well-known result of Hochster and Huneke  \cite[Thm.~2.3]{HH} says that a pure subring of an equal characteristic regular ring is Cohen-Macaulay. We
know from \cite{HO} (resp. \cite{Yanagi2}) that pure subrings of normal (seminormal) rings are normal (resp.~ seminormal). However, compared
to pure subrings, algebra retracts of a ring enjoy much better properties. It is proved in \cite{Cos} that algebra retracts of a regular ring are regular. We give a homological proof of this result in Section \ref{ascent_and_descent}, and prove similar results for local complete intersections, applying ideas of Herzog \cite{Her} and Apassov \cite{Apa}. On the other hand, we give an example showing that the Cohen-Macaulay  and Gorenstein properties are not stable under algebra retracts, see Example \ref{nondescent_Gorenstein}.

In the 2000s, Bruns and Gubeladze proposed many ``polytopal" extensions of classical results in linear algebra \cite{BG99},
\cite{BG99a}, \cite{BG02}. For example, we have the description of the ``polytopal linear group", i.e., the graded automorphism group of a polytopal algebra, which in many ways resembles the general linear group over a field $k$. The general linear group of invertible $n \times n$ matrices over $k$ is exactly the polytopal linear group of the $(n-1)$-simplex; see \cite{BG99} for details. For more discussions of polytopal algebras, see \cite{BGT}. An interesting problem remains open in this program relating to algebra retracts of polytopal algebras:
\begin{conj}[{\cite[Conj.~A]{BG02}}]
\label{polytopal_conjecture}
Every graded algebra retract of a polytopal algebra over $k$ is again a polytopal algebra over $k$.
\end{conj}
This is a generalization of the fact that graded algebra retracts of a standard graded polynomial ring over $k$ are again polynomial rings. Equivalently: every projection of a finite dimension $k$-vector space is diagonalizable to a diagonal matrix with only $0$s and $1$s on the diagonal. Except for some special cases, e.g.~algebra retracts of dimension $2$, the conjecture is wide open in general.

On the other hand, there are many instances where results for affine monoid rings (hence in particular, results for polytopal algebras) have a counterpart for Stanley-Reisner rings. This is the starting point for our consideration of algebra retracts of Stanley-Reisner rings. In fact, we can prove that the analogue of Conjecture \ref{polytopal_conjecture} for Stanley-Reisner rings holds. The main theorem of the paper (Theorem~\ref{graded_retracts_of_SR}) is as follows:
\begin{thm*}
Every graded algebra retract of a Stanley-Reisner ring is a Stanley-Reisner ring.
\end{thm*}
It turns out that the graded algebra retracts of a Stanley-Reisner ring can be described concretely: they correspond to the restrictions of the underlying simplicial complex on subsets of the vertex set. We conjecture a further extension of the main theorem, namely that arbitrary monomial quotient rings also should behave nicely under graded algebra retracts. By definition, a monomial quotient ring over $k$ is a finitely generated $k$-algebra defined by monomial relations. To support this conjecture, we can prove that graded algebra retracts of certain kinds of monomial quotient rings are also monomial quotient rings; see Theorem \ref{irreducible_monomial_quotients} and Proposition
\ref{power_of_linear_ideals}.

The paper is organized as follows. In Section \ref{background_on_retracts}, preliminary materials on algebra retracts, affine monoid rings and Stanley-Reisner rings are recalled. Section \ref{ascent_and_descent} is devoted to algebraic properties of algebra retracts. It is shown that some familiar properties of rings descend along algebra retracts, among them regularity and the complete intersection property. We recall known results of \cite{HHO}, \cite{HO} on the descent of other properties along algebra retracts. At the end of this section, an example of non-descent of Cohen-Macaulayness and Gorensteinness is given. We determine multigraded algebra retracts of toric face rings in Section \ref{toric_face_rings_retracts}. The main theorem of this paper on algebra retracts of Stanley-Reisner rings is proved in Section \ref{base_of_a_retraction}. The last section considers the more general class of monomial quotient rings. The main results in Section~\ref{monomial_quotient_rings} show that
quotients of polynomial rings over $k$ by irreducible monomial ideals or powers of linear ideals are stable under graded
algebra retracts.

Parts of this paper are included in the second named author's dissertation \cite{Nguyen2}.

\section{Background}
\label{background_on_retracts}
For simplicity, all the rings considered in this paper are Noetherian and commutative with unit. A field $k$ is fixed throughout; its characteristic can be either $0$ or positive. The characteristic of the field plays a key role in the proof of Theorem  \ref{irreducible_monomial_quotients}.

The expert should feel free to skip most of the materials in this section and proceed directly to Section \ref{ascent_and_descent}. 

\subsection{Algebra retracts}
\begin{defn}
Let $R\hookrightarrow S$ be an injective ring homomorphism. We say that $R\hookrightarrow S$ is an {\em algebra retract} if there is a
ring homomorphism $\phi: S\to R$ such that the composition map $R\hookrightarrow S \to R$ is the identity. We also say that
$R$ is an algebra retract of $S$. The morphism $\phi$ is called
the {\em retraction map} of the algebra retract $R\hookrightarrow S$. We define \emph{graded} algebra retracts (of graded rings) analogously.
\end{defn}
Observe that the trivial ring $\{0\}$ is {\em never} a retract of a non-trivial ring. Thus in the sequel, we assume that all rings are non-trivial.

Algebra retracts are ubiquitous. Let $R$ be a ring and $R[x]$ the polynomial ring in one variable $x$ over $R$. Clearly the inclusion $R\to R[x]$ is an algebra retract. Another example is
\begin{ex}
\label{tensor_product}
Let $A,B$ be algebras over a field $k$ such that $B$ has an augmentation $\mu:B\to k$. Then we have a natural inclusion
\[
A\hookrightarrow A\otimes_kB, ~ a\mapsto a\otimes 1,
\]
which is an algebra retract with retraction map given by
\[
A\otimes_kB \to A, ~ a\otimes b \mapsto a\cdot \mu(b).
\]
\end{ex}

\subsection{Stanley-Reisner rings and affine monoid rings}
Let $\Delta$ be a set of subsets of $[n]=\{1,\ldots,n\}$ (where $n$ is a positive integer) such that
\begin{enumerate}
\item $\{i\}\in \Delta$ for each $i=1,\ldots,n$.
\item for each $F\in \Delta$, all subsets of $F$ also belong to $\Delta$.
\end{enumerate}
Such a set $\Delta$ is called a {\em simplicial complex} on the vertex set $[n]$.
The elements of $\Delta$ are called its {\em faces}. The maximal faces with respect to inclusion are called {\em facets} of $\Delta$.
Associate with each simplicial complex $\Delta$ a squarefree monomial ideal $I_{\Delta}\subseteq S=k[x_1,\ldots,x_n]$ where
\[
I_{\Delta}=(x^a:a\in \N^n, \supp(a)\notin \Delta),
\]
where the {\em support} $\supp(a)$ of an $n$-tuple $a \in \N^n$ is given by $\supp(a) = \{j \in [n] : a_j \neq 0\}$.
The ring $k[\Delta]=S/I_{\Delta}$ is called the {\em Stanley-Reisner ring} (or {\em face ring}) of the simplicial complex $\Delta$.
\begin{rem}
In the literature, simplicial complexes are sometimes defined without the condition (i). But for a simplicial complexe $\Delta$ on $[n]$ which satisfies only condition (ii), one can change the vertex set suitably to get back a simplicial complex satisfying condition (i) with the same Staneley-Reisner ring. Indeed, if we denote $V=\{i\in [n]: \{i\}\in \Delta\}$, then $\Delta$ satisfies both conditions (i) and (ii) on the vertex set $V$. It is straightforward to check that using the vertex set $V$, the Stanley-Reisner ring of $\Delta$ does not change. Hence the main result of this paper (Theorem \ref{graded_retracts_of_SR}) is still valid with the more flexible definition of simplicial complexes.
\end{rem}
\begin{lem}[{\cite[Thm.~5.1.4]{BH}}]
The {\em unique} irredundant primary decomposition of $I_\Delta$ is given by:
\[
I_{\Delta}=\bigcap_{F~ \text{is a facet of $\Delta$}} (x_i:i\notin F).
\]

\end{lem}

\begin{ex}
\label{StanleyReisnerfaceprojectionmap}
Let $k$ be a field, $\Delta$ a simplicial complex on $[n]$ and $F$ a face of $\Delta$. Denote by $k[F]$ the polynomial ring
$k[x_i:i\in F]$. Then there is an algebra retract
\[
\theta: k[F]\to k[\Delta],
\]
with retraction map the canonical projection
\[
\phi: k[\Delta]\to k[F], ~ x_i\mapsto \begin{cases} x_i & \text{if $i\in F$};\\
                                              0   & \text{otherwise}.

                                \end{cases}
\]
\end{ex}
\begin{defn}
Let $M$ be a finitely generated submonoid of $\Z^d$ for some $d\ge 1$. Then $M$ is called an {\em affine monoid}. The monoid algebra of $M$ over $k$, denoted by $k[M]$, is called an {\em affine monoid ring}.
\end{defn}
Concretely, $k[M]=\bigoplus_{a\in M} k\cdot t^a$ as a $k$-vector space. The multiplication of basis elements of $k[M]$ is given by
\[
t^a\cdot t^b=t^{a+b}
\]
for all $a,b\in M$.

$M$ is called a {\em positive} monoid if the only unit of $M$ is the identity $0$. (Recall that a {\em unit} of a monoid $M$ is an element whose additive inverse is in $M$.)  Let $\R_+M$ be the (polyhedral) cone spanned by $M$ in $\R^d$. Then $M$ is positive if and only if the origin $\{0\}$ is a face of the cone $\R_+M$. Recall that a nonempty subset $F \subseteq C$ is called a {\em face} of a cone $C$ in $\R^d$ if (1) $F$ is the intersection of $C$ with a hyperplane $H$, and (2) $C$ is entirely contained in one of the two half-spaces of $\R^d$ cut out by $H$.  We say that a cone in $\R^d$ is {\em pointed} if $\{0\}$ is a face of it.

\begin{ex}
\label{ex_monoid_retract}
Let $M$ be an affine monoid and $F$ a face of the cone $\R_+M$. There is a natural inclusion map $k[M\cap F] \hookrightarrow k[M]$. Moreover, there
is a surjection $k[M]\to k[M\cap F]$ given by
\[
t^a \mapsto \begin{cases} t^a & \text{if $a\in M\cap F$};\\
                           0   & \text{otherwise}.
\end{cases}
\]
It is easy to check that the composition map $k[M\cap F] \to k[M\cap F]$ is the identity. Thus $k[M\cap F] \hookrightarrow k[M]$ is an algebra retract, with retraction map displayed above.
\end{ex}
For further discussions of Stanley-Reisner rings, affine monoid rings and relevant subjects, we refer to \cite{BH} and \cite{Sta}.
\subsection{Embedded toric face rings}

Let $d\ge 1$, $\Sigma$ be a {\em rational pointed fan} in the euclidean space $\R^d$, with origin $0$. That is, $\Sigma$ is a collection of cones of $\R^d$ such that the following conditions are satisfied for all $C,D \in \Sigma$:
\begin{enumerate}
 \item $C$ is a pointed cone and $C$ is generated by rational vectors;
\item if $C\in \Sigma$ and $C'$ is face of $C$ then $C'\in \Sigma$;
\item $C\cap D$ is either empty or a common face of $C$ and $D$.
\end{enumerate}
Hence ``rational" means each cone of $\Sigma$ is generated by rational vectors, and ``pointed" means each cone of $\Sigma$ is pointed.

A {\em monoidal complex} $\Mcc$ supported on $\Sigma$ is a collection of affine monoids $M_C$, where the parameter $C$ is a varying cone of $\Sigma$, such that for all $C, D\in \Sigma$:
\begin{enumerate}
 \item $M_C\subseteq C\cap \Z^d$ and $M_C$ generates the cone $C$;
\item if $D\subseteq C$, then $M_D=M_C\cap D$.
\end{enumerate}

The {\em toric face ring} of $\Mcc$, denoted by $k[\Mcc]$, is a kind of generalized monoid algebra. To be precise, $k[\Mcc]$ is the $k$-vector space with basis $\{t^a: a\in |\Mcc|=\cup_{C\in \Sigma}M_C\}$, and the multiplication on basis elements is given by
\[
t^a\cdot t^b=\begin{cases}
              t^{a+b}, ~\text{if $a$ and $b$ are contained in $M_C$ for some $C\in \Sigma$,}\\
              0, ~\text{otherwise}.
             \end{cases}
\]
The ring $k[\Mcc]$ is naturally equipped with a $\Z^d$-grading coming from the embedding of $\Sigma$.

We note that toric face rings are a common generalization of Stanley-Reisner rings and affine monoid rings. On the one hand, if each cone of $\Sigma$ is generated by linearly independent vectors, and $M_C$ is generated by exactly $\dim C$ elements for each cone $C$, then $k[\Mcc]$ is a Stanley-Reisner ring. On the other hand, if $\Sigma$ has only one maximal cone $C$, then $k[\Mcc]$ equals the affine monoid rings $k[M_C]$. 

We can compute the defining ideal of a toric face ring in the following way. Let $a_1,\ldots,a_n$ be a set of generators of $\Mcc$, i.e., $a_i\in |\Mcc|$ for every $i$ and $\{a_1,\ldots,a_n\}\cap M_C$ is a set of generators for $M_C$ for every $C\in \Sigma$. Then the generators $a_1,\ldots,a_n$ give rise to a surjection $\pi: k[x_1,\ldots,x_n]\to k[\Mcc]$ mapping $x_i$ to $t^{a_i}$. The defining ideal $I_{\Mcc}=\Ker \pi$ is computed explicitly by
\begin{lem}[{\cite[Prop.~2.3]{BKR}}]
\label{def-ideal}
The defining ideal $I_{\Mcc}$ of $k[\Mcc]$ is generated by the following monomials and binomials of $k[x_1,\ldots,x_n]$:
\begin{enumerate}
\item the monomials of the type $\prod_{i\in H}x_i$ where $H\subseteq [n]$, the set $\{a_i:i\in H\}$ is not contained in any monoid $M_C$ of the monoidal complex $\Mcc$;
\item the binomials of the type $\prod_{i\in H}x_i^{\eta_i}-\prod_{j\in G}x_j^{\nu_j}$, where $H, G\subseteq [n]$, the set $\{a_t: t\in H\cup G\}$ is contained in some monoid $M_C, C\in \Sigma$ and $\sum_{i\in H} \eta_ia_i=\sum_{j\in G} \nu_ja_j$.
\end{enumerate}
\end{lem}
We refer the reader to \cite{BG09}, \cite{BKR}, \cite{IR}, \cite{OkaYan} for deeper discussions of toric face rings.
\subsection{Basic properties of algebra retracts}
\begin{rem}
Observe that if $A\hookrightarrow B$ and $B\hookrightarrow C$ are algebra retracts, then the composition $A\hookrightarrow C$ is also an algebra retract.
\end{rem}
Let $R$ be a subring of $S$. Then $R$ is called a {\em pure subring} of $S$ if for every $R$-module $M$, the $R$-linear morphism $M=M\otimes_R R \to M\otimes_R S$ is injective. It is easy to see that if $R\hookrightarrow S$ is an algebra retract then $R$ is an $R$-direct summand of $S$. In particular, $R$ is a pure subring of $S$.

The following proposition helps us to pass from an arbitrary algebra retract to algebra retracts of local rings.

\begin{prop}
\label{passingtolocalretracts}
Let $R\hookrightarrow S$ be an algebra retract with retraction map $\phi: S\to R$. Then for every prime ideal $\pp \in \Spec R$, there is a natural algebra retract $R_{\pp} \hookrightarrow S_{\qq}$ where $\qq=\phi^{-1}(\pp)$ is a prime ideal of $S$.
\end{prop}

\begin{proof}
This follows from the fact that the algebra retraction $R\hookrightarrow S \to R$ induces a retraction of the underlying spaces of affine schemes $\Spec R \hookrightarrow \Spec S \to \Spec R$.
\end{proof}


\section{Ascent and descent along algebra retracts}
\label{ascent_and_descent}
In this section, we consider the problem: does every algebra retract of a ring with a given property $P$ have property $P$? For the convenience of the reader, we also recall known results on this problem.

A reduced ring $R$ is {\em seminormal} if for all $x,y\in R$ such that $x^2=y^3$, there exists an element $z\in R$ such that $x=z^3$ and $y=z^2$; see \cite{Swa}.
First, note that every algebra retract of a seminormal (normal) ring is also a seminormal (respectively, normal) ring. In fact, a more general fact is true.

\begin{prop}[\cite{HO}, \cite{Yanagi2}]
\label{seminormalretract}
Every pure subring of a seminormal (resp.~ normal) ring is also a seminormal (resp.~ normal) ring.
\end{prop}

The first part follows from the proof of \cite[Prop.~5]{Yanagi2}, since the author only needs the assumption of purity for the proof. The second part of this
result was proved in \cite[Cor.~9.11]{HO}.

It is proved in \cite{Cos} that basic properties like normality of domains, factoriality, and regularity also descend. The
reader may wish to consult \cite{Apa} for more discussion of results of this type, for example the ascent and descent of
Cohen-Macaulayness and Gorensteinness along local morphisms of local rings. In that paper, it is crucial to apply
 relative notions like flat dimension, Cohen-Macaulayness, and Gorensteinness of morphisms.

Behavior of Koszul algebras along graded algebra retracts is given by the following result of Herzog, Hibi, and Ohsugi \cite[Prop.~1.4]{HHO}: Let $R\hookrightarrow S$ be an algebra retract of homogeneous $k$-algebras with  retraction map $\phi: S\rightarrow R$, where $R\neq 0$. Then $S$ is Koszul if and only if $R$ is Koszul and  $R$ has linear resolution as an $S$-module via $\phi$.

We have a similar result for retracts of regular rings and locally complete intersections. The descent of regularity along algebra
retracts was first proved in \cite[Cor.~1.11]{Cos} by a non-homological method. A homological proof of this result is given in the sequel.
\begin{thm}
\label{regularretract}
Every algebra retract of a regular (resp.~locally complete intersection) ring is also a regular (resp.~locally complete intersection) ring.
\end{thm}
Therefore, algebra retracts are much better behaved than direct summands or pure subrings.
It is known that a direct summand of a regular ring is not Gorenstein in general
(e.g. the subring $k[x^4, x^3y, x^2y^2, xy^3, y^4]$ of $k[x,y]$, $k$ an arbitrary field).

\begin{proof}[\it{Proof of Theorem \ref{regularretract}}]
By Proposition \ref{passingtolocalretracts}, it is enough to assume that $R \hookrightarrow S$ be an algebra retract of local rings.
Then applying the following Theorem \ref{regularity_ci_descent}, we obtain the conclusion.
\end{proof}
The case of Theorem~\ref{regularretract} for local rings and homomorphisms follows from the next result.
\begin{thm}[{Apassov \cite[pp.~929-930]{Apa}}]
\label{regularity_ci_descent}
Let $R\to S$ be a local homomorphism of Noetherian local rings. Assume that there exists a finite $S$-module $M$ such that $M$ has finite flat
dimension as an $R$-module. If $S$ is regular (resp.~a complete intersection) then so is $R$.
\end{thm}

Recall that an $R$-module $M$ has finite flat dimension if there is a finite resolution of $M$ by flat $R$-modules. In the case $M$ is a finite
$R$-module, $M$ has finite flat dimension over $R$ if and only if it has finite projective dimension over $R$.

To use Theorem \ref{regularity_ci_descent}, we need only choose $M$ to be the finite $S$-module $R$ (the module structure induced by the retraction map).  Clearly $R$ has finite flat dimension over $R$.

We can characterize regularity completely for local algebra retracts in the next result; the proof depends on a result of Herzog \cite[Thm.~1]{Her}.
\begin{prop}
\label{regularlocalretract}
Let $(R,\mm)\hookrightarrow (S,\nn)$ be an algebra retract of local rings. Then $S$ is regular if and only if $R$ is regular and
$\projdim_S R<\infty$, where $R$ is considered as an $S$-module via the retraction map.
\end{prop}
\begin{proof}
Since there is a surjection $S\to R$, we have $S/{\nn}\cong R/{\mm}$. Let $k$ be the common residue field of $R$ and $S$. By \cite[Thm.~1]{Her}, we have $P^S_k(t)=P^S_R(t)P^R_k(t)$, where $P^R_M(t)$ denotes the Poincar\'e series of a finitely generated $R$-module $M$. In particular, $\projdim_S k=\projdim_S R+\projdim_R k$. By Auslander-Buchsbaum-Serre's theorem, $S$ is regular if and only if $\projdim _S k < \infty$. Using the above identity, the last inequality is equivalent to the condition that $\projdim_S R$ and $\projdim_R k$ are finite. Thus $S$ is regular if and only if $R$ is regular and $\projdim_S R<\infty.$
\end{proof}

\begin{rem}
Consider the ring $S=k[x]/(x(x-1)^2)$ where $k$ is a field, $x$ an indeterminate over $k$, and $R=S/(x) \cong k$. It is easy to see that $R$ is an algebra retract
of $S$ and $R$ is a projective $S$-module. The last statement can be checked by localizing at the maximal ideals of $S$, namely $(x)$ and $(x-1)$ (namely, $R_{(x)} = S_{(x)}$ and $R_{(x-1)} = 0$).
Thus $\projdim _S R=0$. However, $S$ is not a regular ring, since $S_{(x-1)} \cong k[y]/(y^2)$ is not a regular local ring.

Thus the ``if" part of Proposition \ref{regularlocalretract} is not true for non-local rings, namely given an algebra retract of arbitrary rings
$R\hookrightarrow S$, it can happen that $R\neq 0$ is a regular ring and $\projdim_S R < \infty$ but $S$ is a non-regular ring.
\end{rem}
We are going to give a similar characterization of the ascent-descent of complete intersections along local algebra retracts. First, given a noetherian local ring $(R,\mm,k)$ and a finitely generated $R$-module $M$, denote by $\beta^R_i(M)=\dim_k\Tor^R_i(k,M)$ the $i$th Betti number of $M$. The {\em complexity} of $M$ over $R$, denoted $\cx_R M$, is defined as follow
\begin{gather*}
\cx_R M=\inf \left\{d\in \N: \beta^R_i(M)\le ci^{d-1} ~\textnormal{for all $i\gg 0$ and for some constant $c$}\right.\\
\left.\textnormal{not depending on $i$}\right\}.
\end{gather*}
Complete intersections can be characterized by the following analogue of Auslander-Buchsbaum-Serre's theorem.
\begin{thm}[Gulliksen {\cite[Thm.~2.3]{Gu}}]
\label{thm_Gulliksen}
Let $(R,\mm,k)$ be a noetherian local ring. The following statements are equivalent:
\begin{enumerate}
\item $R$ is a complete intersection;
\item $\cx_R M<\infty$ for every finitely generated $R$-module $M$;
\item $\cx_R k<\infty$.
\end{enumerate}
\end{thm}
We are ready for
\begin{prop}
\label{retracts_ci}
Let $(R,\mm)\to (S,\nn)$ be an algebra retract of noetherian local rings. Then $S$ is a complete intersection if and only if $R$ is a complete intersection and $\cx_S R<\infty$, where $R$ is considered as an $S$-module via the retraction map.
\end{prop}
\begin{proof}
Firstly $S/\nn \cong R/\mm =k$. By \cite[Thm.~1]{Her}, we have 
\[
P^S_k(t)=P^S_R(t)P^R_k(t).
\]
Therefore $\max\{\cx_S R,\cx_R k\}\le \cx_S k \le \cx_S R+\cx_R k$. The conclusion follows immediately from Theorem \ref{thm_Gulliksen}.
\end{proof}
\begin{rem}
Complete intersections can also be characterized in terms of {\em curvature}; see Avramov's monograph \cite[Sect.~4.2, Cor.~8.2.2]{Avr}. In the situation of Proposition \ref{retracts_ci}, we can also prove that the following are equivalent:
\begin{enumerate}
\item $S$ is a complete intersection;
\item $R$ is a complete intersection and $\curv_S R\le 1$, where $\curv_S R$ denotes the curvature of $R$ as an $S$-module.
\end{enumerate} 
We leave the details to the interested reader.
\end{rem}
Next, we give an example of non-descent of Cohen-Macaulayness and Gorensteinness. In fact, we can even construct a Gorenstein Stanley-Reisner ring with a non-Cohen-Macaulay algebra retract.
\begin{ex}
\label{nondescent_Gorenstein}
Let $\Delta$ be the simplicial complex on 5 vertices $\{x,y,z,t,w\}$ which corresponds to a pentagon. In other words, the facets
 of $\Delta$ are $\{x,y\}$, $\{y,z\}$, $\{z,t\}$, $\{t,w\}$, $\{w,x\}$.
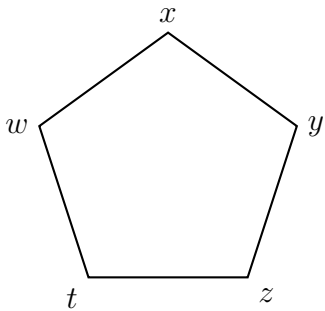
\bigskip

{\centerline{
\begin{tikzpicture}[scale=1.8]
\draw[thick] (0,1)--(-0.9510565163,0.309017)--(-0.58778525229,-0.809017)--(0.58778525229,-0.809017)--(0.9510565163,0.309017)--cycle;
\node [above] at (0,1) {$x$};
\node [right] at (0.9510565163,0.309017) {$y$};
\node [below right] at (0.58778525229,-0.809017) {$z$};
\node [below left] at (-0.58778525229,-0.809017) {$t$};
\node [left] at (-0.9510565163,0.309017) {$w$};
\end{tikzpicture}
}
}
\bigskip

It is known that $S=k[\Delta]$ is a Gorenstein ring, see \cite[Sect.~5.6]{BH}. We have $S=k[x,y,z,t,w]/(xz,xt,yt,yw,zw)$.

Denote $W=\{x,z,t\}$ and consider $R=k[\Delta_W]=k[x,z,t]/(xz,xt)$ (for the notation $\Delta_W$, cf.~ the beginning of Section \ref{toric_face_rings_retracts}). Then by \eqref{StanleyReisnerretract}, $R$ is an algebra retract of $S$.
However, $R$ is not Cohen-Macaulay.
\end{ex}

\section{Multigraded algebra retracts of toric face rings}
\label{toric_face_rings_retracts}

Now we describe all the $\Z^d$-graded algebra retracts of an embedded toric face ring. Let $\Sigma$ be a rational pointed fan in $\R^d$ (where $d\ge 1$) and $\Mcc$ a monoidal complex supported on $\Sigma$.

Given a simplicial complex $\Delta$ with vertex set $[n]$. For each subset $W$ of $[n]$, the restriction $\Delta_W$ of $\Delta$ on $W$ is defined to be the subcomplex
\[
\Delta_W=\{D\in\Delta: D\subseteq W\}.
\]
Generalizing this notion to monoidal complexes, we have the following.

\begin{defn}
Let $\Gamma$ be a subfan of $\Sigma$. We say that $\Gamma$ is a {\em restricted subfan} of $\Sigma$ if for every finite set of elements $x_1,\ldots,x_n$ in $|\Gamma| =\cup_{D\in \Gamma}D$ with the property that $x_1,\ldots,x_n\in C$ for some cone $C\in \Sigma$, we can also find a cone
$D\in \Gamma$ such that $x_1,\ldots,x_n \in D$.
\end{defn}

If $\Sigma$ is a simplicial fan that gives rise to a simplicial complex $\Delta$, it is easy to check that the restricted subfans of $\Sigma$ gives rise to the restrictions of the simplicial complex $\Delta$.

\begin{ex}
Consider the points in $\R^3$ with the following coordinates: $O =(0,0,0), x = (2,0,0),
y=(0,2,0), z = (0,0,2), t = (1,1,0).$

Consider the rational pointed fan $\Sigma$ in $\R^3$ with the maximal cones $\normalfont {Oxy}$ and
$\normalfont{Oyz}$. Let $\Mcc$ be the monoidal complex supported on $\Sigma$ with the two maximal monoids generated by $\{x, y, t\}$ and $\{y,z\}$. 

Applying Lemma \ref{def-ideal}, the toric face ring of $\Mcc$ is
$$
k[\Mcc]=k[x,y,z,t]/(xy-t^2, xz, tz).
$$
\bigskip
\bigskip
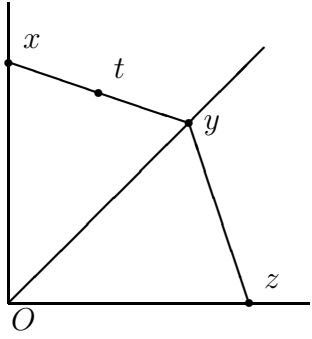
\setlength{\unitlength}{4cm}
\begin{picture}(1,1)
\thicklines
\put(1.5,0){\line(0,1){1}}
\put(1.5,0){\line(1,0){1}}
\put(1.5,0){\line(1,1){0.85}}

\put(1.51,-0.09){$O$}
\put(1.55,0.85){$x$}
\put(1.5,0.8){\circle*{0.03}}
\put(2.35,0.05){$z$}
\put(2.3,0){\circle*{0.03}}
\put(2.15,0.58){$y$}
\put(2.1,0.6){\circle*{0.03}}
\put(1.85,0.75){$t$}
\put(1.8,0.7){\circle*{0.03}}

\put(1.5,0.8){\line(3,-1){.6}}
\put(2.3,0){\line(-1,3){.2}}
\end{picture}
\bigskip
\bigskip

Note that the fan with two one dimensional maximal cones $Ox$ and $Oz$ is a restricted subfan of $\Sigma$. Moreover, we have an algebra retract of $k[\Mcc]$ given by
\[
k[x,z]/(xz) \hookrightarrow k[x,y,z,t]/(xy-t^2, xz, tz)
\]
where the retraction
$$
k[x,y,z,t]/(xy-t^2, xz, tz) \to k[x,z]/(xz)
$$
is the projection mapping $y$ and $t$ to zero.
\end{ex}

\begin{ex}
Let $M \subseteq \N^d$ be an positive affine monoid and $\R_+M$ the cone spanned by $M$. Then a subfan $\Gamma$ of $\R_+M$ is restricted if and only if $\Gamma$ is the subfan associated with a face $F$ of $\R_+M$.
\end{ex}

\begin{lem}
\label{lem_restricted}
Let $\Mcc$ be a monoidal complex supported on the rational pointed fan $\Sigma$ in $\R^d$ (where $d\ge 1$). Let $\Gamma$ be a subfan of $\Sigma$ and $\Mcc_{\Gamma}$ the induced monoidal subcomplex on $\Gamma$. The following statements are equivalent:
\begin{enumerate}
\item $\Gamma$ is a restricted subfan of $\Sigma$;
\item for any cones $C, C'\in \Gamma$ such that $C\cup C'\subseteq C''$ for some cone $C''\in \Sigma$, there exists a cone $D\in \Gamma$ such that $C\cup C'\subseteq D$;
\item for any $x_1,\ldots,x_n\in |\Mcc_{\Gamma}|$ such that $x_1,\ldots,x_n\in M_C$ for some cone $C\in \Sigma$, there exists a cone $D\in \Gamma$ such that $x_1,\ldots,x_n\in M_D$.
\end{enumerate}
\end{lem}
\begin{proof}
(i) $\Rightarrow$ (ii): it is enough to take a system of generators $x_1,\ldots,x_n$ of $C\cup C'$ as cones.

(i) $\Rightarrow$ (iii): since $\Gamma$ is restricted subfan of $\Sigma$, we can find $D\in \Gamma$ such that $x_1,\ldots,x_n\in D$. This implies that $D\cap C\in \Gamma$ and $x_i\in D\cap C \cap M_C=M_{D\cap C}$ for all $i=1,\ldots,n$.

(ii) $\Rightarrow$ (i): let $x_1,\ldots,x_n \in |\Gamma|$ be such that $x_1,\ldots,x_n\in C$ for some cone $C\in \Sigma$. We have to show that $x_1,\ldots,x_n\in D$ for some $D\in \Gamma$. Firstly, we can choose cones $D_1,\ldots,D_m \in \Gamma$ such that $x_1,\ldots,x_n \in D_1 \cup \cdots \cup D_m$. As $x_1,\ldots,x_n\in C$, replacing $D_j$ by $D_j\cap C$, we can assume that $D_j\subseteq C$. If $m=1$, we are done. Otherwise since $D_{m-1}\cup D_m \subseteq C$, by (ii), we can find a cone $D'_{m-1}\in \Gamma$ such that $D_{m-1}, D_m\subseteq D'_{m-1}$. Further replacing $D'_{m-1}$ by $D'_{m-1}\cap C$ then we have $x_1,\ldots,x_n\in D_1 \cup \cdots \cup D_{m-2}\cup D'_{m-1}$ and $D_1,\ldots,D_{m-2},D'_{m-1}\subseteq C$. Continuing in this manner, finally we find some $D\in \Gamma$ so that $x_1,\ldots,x_n\in D$.

(iii) $\Rightarrow$ (i): again let $x_1,\ldots,x_n \in |\Gamma|$ be such that $x_1,\ldots,x_n\in C$ for some cone $C\in \Sigma$. Let $D_1,\ldots,D_n$ be the (not necessarily distinct) cones of $\Gamma$ such that $x_i$ is contained in the relative interior of $D_i$ for $i=1,\ldots,n$. Then $D_i\subseteq C$ for each $i$. By continuity, for each $i\in [n]$, we can choose $y_i\in D_i\cap \Q^d$ such that $y_i$ is contained in the relative interior of $D_i$. 

Since $C$ is a rational cone generated by $M_C$, replacing $y_1,\ldots,y_n$ by suitable integral multiples of them, we can assume that $y_1,\ldots,y_n\in M_C$. By (iii), there is a cone $D\in \Gamma$ such that $y_1,\ldots,y_n\in M_D$. In particular, $y_1,\ldots,y_n\in D$, and consequently $D_1,\ldots,D_n\subseteq D$. Hence we get the desired conclusion as $x_1,\ldots,x_n\in D$.
\end{proof}
\begin{prop}
The ring $A$ is a $\Z^d$-graded algebra retract of $k[\Mcc]$ if and only if there is a restricted subfan $\Gamma$ of $\Sigma$ such that $A\cong k[\Mcc_{\Gamma}].$
\end{prop}

\begin{proof}
For any subfan $\Gamma$ of $\Sigma$, since $|\Mcc_{\Gamma}|\subseteq |\Mcc|=|\Mcc_{\Sigma}|$, there is a natural $k$-linear inclusion of $k$-vector spaces $\iota: k[\Mcc_{\Gamma}]\to k[\Mcc]$. On the other hand, assume that $A$ is a $\Z^d$-graded algebra retract of $k[\Mcc]$. The ring $A$ is reduced since it is a subring of the reduced ring $k[\Mcc]$. Moreover, since $A$ is a $\Z^d$-graded quotient of $k[\Mcc]$, by \cite[Lem.~2.1]{IR}, there is a subfan $\Gamma'$ of $\Sigma$ such that $A\cong k[\Mcc_{\Gamma'}].$ The inclusion $k[\Mcc_{\Gamma'}]\cong A \to k[\Mcc]$ is identical with the inclusion $\iota$ defined above. Hence it suffices to show that $\iota$ respects multiplication if and only if $\Gamma$ is a restricted subfan of $\Gamma$. Since there is no danger of confusion, by abuse of notation, for any $a\in |\Mcc_{\Gamma}|$, we also denote $t^a\in k[\Mcc_{\Gamma}]$ by $a$.

If $\Gamma$ is a restricted subfan, take any elements $x_1,\ldots,x_n\in |\Mcc_{\Gamma}|$. If $x_1,\ldots,x_n\in D$ for some $D\in \Gamma$ then clearly $\iota$ sends $x_1\cdots x_n \in k[\Mcc_{\Gamma}]$ to the right element in $k[\Mcc]$. If on other hand there exists no such cone $D\in \Gamma$ then $x_1\cdots x_n=0$ in $k[\Mcc_{\Gamma}]$. As $\Gamma$ is restricted, there is also no $C\in \Sigma$ such that $x_1,\ldots,x_n\in C$, hence $x_1\cdots x_n=0$ in $k[\Mcc]$ too. In any case $\iota$ does respect multiplication.

If $\Gamma$ is not a restricted subfan, then by Lemma \ref{lem_restricted}, there are elements $x_1,\ldots,x_n$ in $|\Mcc_{\Gamma}|$ such that $x_1,\ldots,x_n\in M_C$ for some cone $C\in \Sigma$ but there is no cone $D\in \Gamma$ such that $x_1,\ldots,x_n\in M_D$. Thus $x_1\cdots x_n=0$ in $k[\Mcc_{\Gamma}]$ but $x_1\cdots x_n \neq 0$ in
$k[\Mcc]$, so $\iota$ does not respect multiplication in this situation.
\end{proof}

 As consequences, we can easily classify all the $\Z^n$-graded algebra retracts of  affine monoid rings and Stanley-Reisner rings.
\begin{cor}
Let $M\subseteq \N^n$ be a positive affine monoid. Then $R$ is a $\Z^n$-graded algebra retract of $k[M]$ if and only if $R\cong k[M\cap F]$ for some face $F$ of $\R_+M$.
\end{cor}
\begin{cor}
\label{StanleyReisnerretract}
Let $\Delta$ be a simplicial complex on $[n]$ and $k[\Delta]$ the corresponding Stanley-Reisner ring. Then $R$ is a $\Z^n$-graded algebra retract of $k[\Delta]$ if and only if $R\cong k[\Delta_W]$ for some subset $W$ of $[n]$.
\end{cor}

\section{Bases and Stanley-Reisner retractions}
\label{base_of_a_retraction}

In \cite[Conj.~A]{BG02}, the authors conjecture that graded algebra retracts of polytopal algebras over $k$ are again polytopal algebras
over $k$. Motivated by this conjecture and by the close relationship between Stanley-Reisner rings and affine monoid
rings, we prove the main result of this paper.

\begin{thm}
\label{graded_retracts_of_SR}
Every $\Z$-graded algebra retract of a standard graded Stanley-Reisner ring over $k$ is a Stanley-Reisner ring over $k$.
\end{thm}

First we introduce some more notions.
\begin{defn}
Let $R\hookrightarrow S$ be an algebra retract with retraction map $\phi: S\to R$. An ideal $I$ of $S$ is called a {\em compatible} ideal if $\phi(I) \subseteq
I$ (equivalently $\phi(I)=I\cap R$; note that $\phi$ restricts to the identity on $R$).
\end{defn}
Clearly if $I$ is a compatible ideal of the algebra retract $R\hookrightarrow S$ then there is an induced algebra retract
$R/\phi(I) \to S/I.$

Now assume that $S=k[x_1,\ldots,x_n]$ is a standard graded polynomial ring over $k$ and denote $\nn=S_+$. Let $I$ be a monomial ideal of $S$. For each subset $W$ of $[n]$, let $S_W:=k[x_i:i\in W]$, and define $I_W$ to be the ideal of $S_W$ generated by monomials of $I$ that are products of variables in $S_W$.

Let $\theta: A\to S/I$ be a graded algebra retract of standard graded $k$-algebras, where $S$ and $I$ are as above. Let $A=R/J$ be a presentation of $A$, where $R$ is a standard graded polynomial ring over $k$ and $J$ is a graded ideal of $R$. Denote $\mm=R_+$. The following lemma is immediate using degree reasoning.
\begin{lem}
\label{induced_polynomial_retract}
With the above notations, if $I\subseteq \nn^2$ and $J\subseteq \mm^2$, then the algebra retract $R/J \to S/I$ induces a graded algebra retract of polynomial rings $R\to S$ with retraction map $\phi: S\to R$ such that $\phi(I)=J= I\cap R$.
\end{lem}

Assume that the conclusion of the above lemma is satisfied for some presentation $R/J$ of $A$ (this is automatic if $I\subseteq \nn^2$). That is, there is a graded algebra retract $R\to S$ with retraction map $\phi: S\to R$ such that $\phi(I)=J$. Denote by $\supp(I)$ the set of variables of $S$ which divide a minimal generator of $I$. With these notations, and motivated by the notion in \cite{BG02} of a ``based retraction'', we define a base of the algebra retract $\theta$ as follows.
\begin{defn}[Base of a retract]
\label{defn_base}
A subset $W\subseteq [n]$ is called a {\em base} of the algebra retract $\theta: R/J \to S/I$ if the two conditions belowe are satisfied:
\begin{enumerate}
\item $\{\phi(x_i):i\in W\}$ are linearly independent over $k$ (as elements of $R_1$),
\item $\phi(I)=\phi(I_WS)$.
\end{enumerate}
\end{defn}
Let us look at some examples.
\begin{ex}
Let $S=k[x,y,z]$ and $I=(xy,yz,zx)$. Consider the algebra retract $\theta: k[x,y]/(xy) \hookrightarrow k[x,y,z]/(xy,yz,zx)$, with the retraction map $\phi$ fixing $x,y$ and setting
$\phi(z)=0$. Then $\{x,y\}$ is a base of $\theta$.
\end{ex}
\begin{ex}
Let $\chara k=5$ and $S=k[x,y,z], I=(x^5,y^5,z^5)$. Consider the inclusion $\theta: k[t]/(t^5) \to S/I$ mapping $t$ to $x+y$. Let $\phi: S/I \to k[t]/(t^5)$ be given by:
\[
\phi(x)=\phi(y)=\phi(z)=\frac{t}{2} = 3t.
\]
Then $\theta$ is a retract with retraction map $\phi$, and $\{z\}$ is a base of $\theta$.
\end{ex}

A simple observation: if $W$ is a base of the algebra retract $R/J \to S/I$ then replace $W$ by a larger subset of $[n]$ if necessary, we have $R\cong S_W$ and $J=\phi(I)=\phi(I_WS)\cong I_W$. Therefore $R/J \cong S_W/I_W$, a monomial quotient ring. If $I$ is squarefree then clearly $I_W$ is also squarefree, hence like $S/I$, $R/J$ is isomorphic to a Stanley-Reisner ring. Therefore to prove Theorem \ref{graded_retracts_of_SR}, it is enough to confirm
\begin{thm}
\label{base_SR}
Every graded algebra retract of a Stanley-Reisner ring has a base.
\end{thm}
{\bf Sketch of the proof}: Let $I$ be the ideal defining our Stanley-Reisner ring. Let $I=\pp_1\cap \cdots \cap \pp_s$ be the primary decomposition of $I$. From Lemma \ref{induced_polynomial_retract}, each graded retract of $S/I$ induces a graded retract of polynomial rings $R\hookrightarrow S$. The first step is to  reduce to the case where each of the $\pp_i, i=1,\ldots,s$, is a compatible ideal of the induced algebra retract $R\to S$. The second step is choosing a base for each of the induced algebra retracts corresponding to $\pp_i, i=1,\ldots,s$. If we choose these bases carefully enough, then their {\em union} will be a base for the original retract of $S/I$.

Before going to the proof of \ref{base_SR}, we prove a useful lemma, which was suggested by the referee. We are not aware of any occurrence of this result in the literature. 
\begin{lem}
\label{lem_extension}
Let $f:A\to B$ be a homomorphism of noetherian rings and $I$ a radical ideal of $A$. Let $I=\pp_1\cap \cdots \cap \pp_r$ be the primary decomposition by associated prime ideals of $I$. Assume that $IB$ is a radical ideal of $B$. Then
\[
IB=\pp_1B\cap \cdots \cap \pp_rB.
\]
\end{lem}
\begin{proof}
Let $IB=\qq_1\cap \cdots \qq_{r'}$ be the primary decomposition by associated prime ideals of $IB$. Then 
\[
\pp_1\cap \cdots \cap \pp_r \subseteq IB \cap A = f^{-1}(\qq_1)\cap \cdots \cap f^{-1}(\qq_{r'}).
\]
Therefore for each $j\in \{1,\ldots,r'\}$, there exists $1\le i\le r$ such that $\pp_i \subseteq f^{-1}(\qq_j)$. In particular, we obtain the second inclusion in the following display
\[
IB \subseteq  \pp_1B\cap \cdots \cap \pp_rB \subseteq \qq_1\cap \cdots \cap \qq_{r'}=IB.
\]
This concludes the proof.
\end{proof}
\begin{proof}[Proof of Theorem \ref{base_SR}]
Let $S=k[x_1,\ldots,x_n]$ be a standard graded polynomial ring. Let $I$ be a squarefree monomial ideal of $S$. If $n=0$ or $I=0$, there is nothing to do, so we assume that $n\ge 1$ and $I\neq 0$. Note that $I \subseteq \nn^2$ where $\nn=(x_1,\ldots,x_n)$.

Assume that $A$ is a graded algebra retract of $S/I$. Since $A$ is a quotient of $S$, we can write $A=R/J$ where
$R=k[y_1,\ldots,y_m]$ is a standard graded polynomial ring and $J\subseteq (y_1,\ldots,y_m)^2$. We call this the minimal
presentation of $A$. Of course, a minimal presentation of a standard graded $k$-algebra exists and is unique in the sense that
if $R'/J'$ is another minimal presentation then there is an isomorphism $R\cong R'$ mapping $J$ to $J'$.

Obviously, the graded algebra retract $R/J \hookrightarrow S/I \to R/J $ induces a graded retract of polynomial rings
$R\hookrightarrow S \to R$. Thus, we may assume that $y_1,\ldots,y_m \in S_1$. From Lemma \ref{induced_polynomial_retract},  $J=\phi(I)=I\cap R$.

Let the irredundant primary decomposition of $I$ be $I=\pp_1\cap \cdots\cap \pp_s$ where $\pp_i$ are linear ideals
associated to $I$.

{\em Step 1}. Since $S/I$ is reduced, we also have $J=\phi(I)$ is a radical ideal of $R$. Using Lemma \ref{lem_extension} for the surjection $\phi:S\to R$ and the ideal $I$, we obtain the equality
\[
\phi(I)=\phi(\pp_1)\cap \cdots \cap \phi(\pp_s).
\]
Moreover, for each $i=1,\ldots,s$ there exists some $j\in \{1,\ldots,s\}$ such that $\pp_i \cap R \supseteq \phi(\pp_j)$. The last claim follows from the fact that $\phi(\pp_1)\cap \cdots \cap \phi(\pp_s)=I\cap R= (\pp_1\cap R)\cap \cdots \cap (\pp_s\cap R)$ and $\pp_i\cap R$ is a prime ideal for each $1\le i\le s$. 

{\em Step 2}. We will use induction on $s\ge 0$ to finish the proof of the theorem.

If $s=0$ then $I=0$, there is nothing to do. Assume that $s\ge 1$. We will prove the following claim.

\begin{claim*}
It is possible to reduce the general case to the case $\phi(\pp_i)=\pp_i\cap R$ for each associated prime $\pp_i$ of $I$.
\end{claim*}

Of course $\pp_i\cap R \subseteq \phi(\pp_i)$ for every $i$. So we assume that $\phi(\pp_i)\nsubseteq \pp_i\cap R$ for some $i$, say $i=1$. By Step 1, there exists $i\neq 1$ such that $\phi(\pp_i)\subseteq \pp_1\cap R$. In particular, $\phi^2(\pp_i)=\phi(\pp_i)\subseteq \phi(\pp_1)$. 

Consider the ideal $I'=\bigcap_{j=2}^s\pp_j$. Now we have
\[
\phi(I') \subseteq \bigcap_{j=2}^s\phi(\pp_j) = \bigcap_{j=1}^s\phi(\pp_j) =\phi(I) \subseteq \phi(I').
\]
The second equality holds since $\phi(\pp_i)\subseteq \phi(\pp_1)$ while the last inclusion holds since $I\subseteq I'$. By the induction hypothesis on $s$, there is a base for the algebra retract
 $R/\phi(I)=R/\phi(I')\to S/I'$. This base is also a base for $\theta$. Hence we can make the reduction in the claim.

{\em Step 3}. Consider the case $\phi(\pp_i)=\pp_i\cap R$ for $i=1,\ldots,s$. From Lemma \ref{glueing_basis} below, there exists a subset
 $L=\{x_{g_1},\ldots,x_{g_p}\}$ of $\supp(I)$ such that:
\begin{enumerate}
\item $\phi(x_{g_1}),\ldots,\phi(x_{g_p})$ are linearly independent over $k$,
\item $\phi(\pp_i)=(\phi(x_{g_j}):x_{g_j} \in \pp_i)$, for $i=1,\ldots,s$.
\end{enumerate}
In fact, we only use the case $t=1$ of condition (ii) in Lemma \ref{glueing_basis}.
 It is not hard to verify that $L$ is a base for $\theta$.

Hence we finish the induction on $s$. This concludes the proof.
\end{proof}
It remains to establish the following lemma.
\begin{lem}
\label{glueing_basis}
Let $R\hookrightarrow S=k[x_1,\ldots,x_n]$ be a graded algebra retract of standard graded polynomial rings over $k$, and  $\phi:S\to R$ a retraction map.

Let $I$ be a squarefree monomial ideal of $S$ with primary decomposition $I=\pp_1\cap \cdots \cap \pp_s$, where $s\ge 1$.
If $\phi(\pp_i)=\pp_i\cap R$ for $i=1,\ldots,s$, then there exists a subset $L=\{x_{g_1},\ldots,x_{g_p}\}$ of $\supp(I)$ such that:
\begin{enumerate}
\item The vectors $\phi(x_{g_1}),\ldots,\phi(x_{g_p})$ are linearly independent over $k$,
\item $\phi(\pp_{i_1\ldots i_t})=(\phi(x_{g_j}):x_{g_j} \in \pp_{i_1\ldots i_t}),$
for every sequence $i_1,\ldots,i_t$ with $1\le i_1<\cdots<i_t\le s$.
\end{enumerate}
Here we denote $\pp_{i_1\ldots i_t}=(x_j: x_j\in \pp_{i_1}\cap \cdots \cap \pp_{i_t})$ for each increasing sequence $i_1,\ldots,i_t$
in $\{1,\ldots,s\}$.
\end{lem}
\begin{proof}
Observe that from the hypothesis, for any sequence $1\le i_1 <\cdots <i_t \le s$, we have
\[
\phi\left(\bigcap_{j=1}^t\pp_{i_j}\right)= \bigcap_{j=1}^t\phi(\pp_{i_j}).
\]
This is true because
\[
\phi\left(\bigcap_{j=1}^t\pp_{i_j}\right) \subseteq \bigcap_{j=1}^t\phi(\pp_{i_j}) = \bigcap_{j=1}^t (\pp_{i_j}\cap R)=\left(\bigcap_{j=1}^t \pp_{i_j}\right)\cap R \subseteq \phi\left(\bigcap_{j=1}^t\pp_{i_j}\right).
\]
We use reverse induction on $t$ with $1\le t \le s$ to show that there exists a subset $L_t$ of $\bigcup_{1\le i_1<\cdots<i_t\le s}\supp(\pp_{i_1\ldots i_t})$ such that:
\begin{enumerate}
\item The vectors $\phi(x_i)$ (where $x_i\in L_t$) are linearly independent over $k$,
\item $\phi(\pp_{i_1\ldots i_r})=(\phi(x_j):x_j \in L_t \cap \pp_{i_1\ldots i_r}),$
for every sequence $i_1,\ldots,i_r$ with $r\ge t$ and $1\le i_1<\cdots<i_r\le s$.
\end{enumerate}
Note that after finishing this induction, we set $L=L_1$ and get the conclusion of the lemma.

Consider the starting case $t=s$. Since $\phi(\pp_i)\subseteq \pp_i$ for $i=1,\ldots,s$, we get $\phi(\pp_{i_1\ldots i_r}) \subseteq
\pp_{i_1\ldots i_r}$ for each sequence $1\le i_1<\cdots < i_r \le s$. In detail, this is because
\[
\phi(\pp_{i_1\ldots i_r}) \subseteq \pp_{i_1}\cap \cdots \cap \pp_{i_r},
\]
and $\phi(\pp_{i_1\ldots i_r})$ is generated by linear forms.

Hence $\phi(\pp_{12\cdots s})\subseteq \pp_{12\cdots s}$. It is enough to choose $L_s$ to be the subset of $\{x_i:x_i\in \pp_{12\cdots s}\}$ such that the vectors $\{\phi(x_i): x_i \in L_s\}$ is a $k$-basis for $\phi(\pp_{12\cdots s})$ in degree $1$.

Assume that the statement is true for $t+1$, so we have a set $L_{t+1}$ with suitable properties. For each increasing sequence of $t$ elements $i_1,\ldots,i_t$, we denote $V_{i_1\ldots i_t}=\Span (\phi(x_j):x_j\in \pp_{i_1}\cap \cdots \cap \pp_{i_t}) \subseteq R_1$. Observe that
\[
V_{i_1\ldots i_t}=\left[\phi\left(\bigcap_{r=1}^t\pp_{i_r}\right)\right]_1=\left[\bigcap_{r=1}^t \phi(\pp_{i_r})\right]_1,
\] 
where for a $\Z$-graded $R$-module $M$, $[M]_1$ denotes the $k$-vector space spanned by elements of degree $1$ of $M$.

Let $B_{i_1\ldots i_t}$ be the set of variables
 in \[
\pp_{i_1\ldots i_t}\setminus \bigcup_{j\notin \{i_1,\ldots,i_t\}} \pp_{j}.
\]
Let $C_{i_1\ldots i_t}$ be a minimal subset of $B_{i_1\ldots i_t}$ such that
$\Span(\phi(x_i):x_i\in C_{i_1\ldots i_t})$ is a complement vector space for
\[
\sum_{\{i_1,\ldots,i_t\}\subsetneq \{j_1,\ldots,j_{t+1}\}} V_{j_1,\ldots,j_{t+1}}
\]
in $V_{i_1\ldots i_t}$.

Let $L_t =\left(\cup_{i_1,\ldots,i_t}  C_{i_1\ldots i_t}\right) \cup L_{t+1} $. It is easy to check that $L_t$ satisfies all the stated requirements. 
\end{proof}
In combination with Corollary \ref{StanleyReisnerretract}, we can classify all $\Z$-graded algebra retracts of standard graded Stanley-Reisner rings.

Given two simplicial complexes $\Delta, \Delta'$ on the vertex sets $V$ and $V'$, a morphism $f$ from $\Delta$ to $\Delta'$ is a map
$f: V \to V'$ such that the image of a face of $\Delta$ is a face of $\Delta'$. We say $\Delta$ and $\Delta'$ are {\em isomorphic simplicial
complexes} if there exist morphisms $f:\Delta \to \Delta'$ and $g: \Delta' \to \Delta$ such that $f:V\to V'$ is a bijection with inverse
 $g:V' \to V$.
\begin{cor}
Let $\Delta$ be a simplicial complex on $[n]$ and $k[\Delta]$ the standard graded Stanley-Reisner ring of $\Delta$. Every $\Z$-graded algebra retract of $k[\Delta]$ is isomorphic as $\Z$-graded $k$-algebras to one of the rings $k[\Delta_W]$, where $W$ is
a subset of $[n]$.
\end{cor}
\begin{proof}
This is clear from Theorem \ref{base_SR}, namely if $W \subseteq \supp(I_{\Delta})$ is a base of a graded
algebra retract $R$ of $S/I_{\Delta}$ then $R\cong k[\Delta_W]$. 
\end{proof}
Note that by \cite[Main Theorem]{BG96}, for two subsets $W_1, W_2$ of $[n]$, the $\Z$-graded $k$-algebras $k[\Delta_{W_1}]$ and $k[\Delta_{W_2}]$ are isomorphic if and only if $\Delta_{W_1}$ and $\Delta_{W_2}$ are isomorphic as simplicial complexes.

\section{Monomial quotient rings}
\label{monomial_quotient_rings}

In this section, we study the graded algebra retracts of monomial quotients of standard graded polynomial rings over $k$. We use the notations of Section
\ref{base_of_a_retraction}. The results in Section \ref{base_of_a_retraction} and in this section support the following conjecture.

\begin{conj}
\label{retracts_of_monomial_quotient_rings}
Let $I$ be a monomial ideal in the standard graded polynomial ring $S=k[x_1,\ldots,x_n]$, where $I\subseteq (S_+)^2$. Then every graded algebra retract of $S/I$ has a base.
\end{conj}

\begin{rem}
If the conjecture were true, then every graded algebra retract of a standard graded monomial quotient ring $S/I$ is isomorphic
to one of the rings $S_W/I_W$ where $W$ is a subset of $[n]$.
\end{rem}

\begin{rem}
For $\Z^n$-graded algebra retracts $R/J\hookrightarrow S/I$ where $R$ is multigraded polynomial subring of $S$, the conjecture is true. The detailed argument is left to the reader.
\end{rem}

We call a monomial ideal $I\neq (1)$ an {\em irreducible monomial ideal} if $I$ is generated by powers of variables. In other words, $I=(x_{i_1}^{d_{i_1}},\ldots,x_{i_s}^{d_{i_s}})$
where $1\le i_1 <\cdots <i_s\le n$ and all $d_{i_j} \ge 1$ (we allow the case where $s=0$ and $I=(0)$). We will prove Conjecture \ref{retracts_of_monomial_quotient_rings} in several
situations, where $I$ is either an irreducible monomial ideal or a power of some linear ideal.

It is worth mentioning that every monomial ideal has a unique irredundant irreducible decomposition. Indeed, it is enough to use the following observation: if $a,b$ are monomials of $S$ with $\gcd(a,b)=1$ and $J$ is a monomial ideal then
\[
(ab)+J =((a)+J) \cap ((b)+J).
\]
From this observation, it is clear that our use of the word ``irreducible'' coincides with standard terminology. The reader may wish to consult \cite[§1.3]{HerH} for a precise discussion of irreducible decompositions of monomial ideals.

Now we prove the main result of this section.
\begin{thm}
\label{irreducible_monomial_quotients}
If $I=(x_{i_1}^{d_{i_1}},\ldots,x_{i_s}^{d_{i_s}})$ (where $1\le i_1 <\cdots <i_s\le n$ and all $d_{i_j} \ge 2$) is an irreducible monomial ideal of $S$ then Conjecture \ref{retracts_of_monomial_quotient_rings} holds for $S$ and $I$.
\end{thm}

\begin{proof}
Without loss of generality, we may assume that $i_1=1,\ldots,i_s=s$ (and $d_1,\ldots,d_s\ge 2$). Hence the irreducible monomial ideal $I$ equals $(x_1^{d_1},\ldots,x_s^{d_s})$.

Let $R/J \hookrightarrow S/I$ be a graded algebra
retract with retraction map $S/I \to R/J$, where $R$ is a standard graded polynomial ring and $R/J$ is in minimal presentation form. Clearly we have an induced retract $R\hookrightarrow S$ with retraction map $\phi: S\to R$. As in Lemma \ref{induced_polynomial_retract}, we know that $J=\phi(I)=I\cap R$. Let $y_i:= \phi(x_i)$ where $i=1,\ldots,n$.

First we reduce the problem to the case $s=n$. Since $x_i \in \sqrt{I}$, we have $y_i \in
\sqrt{\phi(I)}\subseteq \sqrt{I}=(x_1,\ldots,x_s)$ for $i=1,\ldots,s$. Let $I'$ be the ideal $(x_1^{d_1},\ldots,x_s^{d_s})$ of $k[x_1,\ldots,x_s]$.

Choose a complement $k$-basis
$y_{i_1},\ldots,y_{i_p}$ for $\Span(y_1,\ldots,y_s)$ in $R_1$, where $s+1 \le i_1 <\cdots <i_p \le n$. Then we have
\[
R/J = k[y_1,\ldots,y_s]/\phi(I') \otimes_k k[y_{i_1},\ldots,y_{i_p}].
\]
Since $\phi(I) \subseteq I$, we also have $\phi(I') \subseteq I'$. Hence $k[y_1,\ldots,y_s]/\phi(I')$ is
 a $\Z$-graded algebra retract of $k[x_1,\ldots,x_s]/(x_1^{d_1},\ldots,x_s^{d_s})$. A base for this retract is also
a base for the original algebra retract $R/\phi(I) \to S/I$. Hence we can reduce the problem to the case $s=n$.

We proceed by induction on $n$. If $n=0$, there is nothing to do. Consider the case $n\ge 1$.

Without loss of generality, we may assume that $d_n=\ldots =d_{t+1}=d >d_t \ge \cdots \ge d_1$ (where $t\ge 0$). We will show by contradiction that
$\phi(x_j) \in (x_1,\ldots,x_t)$ for $j=1,\ldots,t$. Assume that this is not the case. Then for some $j\le t$, $t+1\le i \le n$, we have
\[
\phi(x_j)=cx_i + \sum_{\ell\neq i}b_{\ell}x_{\ell},
\]
where $c,b_{\ell}\in k, c\neq 0$. Since $x_j^{d_j}\in I$, we get $\phi(x_j)^{d_j} \in I$. In particular, $x_i^{d_j} \in I$ (since $I$ is a monomial ideal). This is a contradiction since $d_j<d_i=d$.

Let $I''$ be the ideal of $S''=k[x_1,\ldots,x_t]$ generated by the elements $x_1^{d_1},\ldots, x_t^{d_t}$. Let $R''=k[y_1,\ldots,y_t]$ be the symmetric algebra of the $k$-vector space $\Span(y_1,\ldots,y_t) \subseteq S''_1$. We have
$\phi(I'') \subseteq I''$, since from above $\phi(x_j) \subseteq (x_1,\ldots,x_t)$ for $j=1,\ldots,t$. In particular, $R''/\phi(I'')$ is
a graded algebra retract of $S''/I''$. Hence by induction hypothesis, this algebra retract possesses a base $U=\{i_1,\ldots,i_g\}\subseteq \{1,\ldots,t\}$, so that
\begin{enumerate}
\item $y_{i_1},\ldots,y_{i_g}$ are linearly independent over $k$;
\item $\phi(I'')=\phi(I''_US'')$.
\end{enumerate}
We consider the following two cases.

{\em Case 1}. If either $\chara k=0$ or $\chara k=p>0$ and $d$ is not a power of $p$, we prove the next claim.
\begin{claim*}
At least one of the following two situations occurs:
 \begin{enumerate}
\item for each $i=t+1,\ldots,n$, there exist $b_1,\ldots,b_t\in k$ such that $\phi(x_i)=x_i+\sum_{j=1}^{t}b_jx_j$ ,
\item for some $i\in \{t+1,\ldots,n\}$, we have $\phi(x_i)=\sum_{j\neq i}c_jx_j$ (where $c_j\in k$).
\end{enumerate}
\end{claim*}

Note that if the claim were proved, then we could finish the proof of the theorem. Indeed, assume that (i) is true. We can choose $W=\{i_1,\ldots,
i_g,t+1,\ldots,n\}$.

On the other hand, if (ii) is true, then for some $t+1\le i \le n$, $\phi(x_i)=\sum_{j\neq i}c_jx_j\in R$, and hence $\phi(x_i-\sum_{j\neq i}c_jx_j)=0$. This implies that
$R/J$ is a graded retract of $S/I+(x_i-\sum_{j\neq i}c_jx_j)$. Denote $W=[n]\setminus \{i\}$. We note that the latter ring is $S_W/I_W$.
Indeed, we only need to observe that $x_i^d\in I$ implies that $(\sum_{j\neq i}c_jx_j)^d \in I$, which in turn implies that
$(\sum_{j\neq i}c_jx_j)^d \in I_W$. Hence $S/I+(x_i-\sum_{j\neq i}c_jx_j)\cong S_W/I_W$.

Now since $R/J$ is a $\Z$-graded retract of $S_W/I_W$, we can apply the induction hypothesis and get the desired conclusion.

We are left with proving the claim. Assume that (ii) fails. There exists some $i \in [n]$ such that
\[
\phi(x_i)=cx_i + \sum_{j\neq i}c_jx_j,
\]
where $c,c_j \in k$, $c\neq 0$. First we prove by contradiction that $c_j=0$ for every $j\in \{t+1,\ldots,n\}\setminus \{i\}$.
Assume that this is not the case, so $c_j\neq 0$ for some $j\in \{t+1,\ldots,n\}\setminus \{i\}$.

If $\chara k=0$, since $x_i^d\in I$, we have $d(cx_i)^{d-1}c_jx_j \in I$. Hence
$x_i^{d-1}x_j\in I$, which is a contradiction.

If $\chara k=p>0$ and $d$ is not a power of $p$, we may write $d=hp^r$ where $r\ge 0$, $h\ge 2$ and $h$ is
not divisible by $p$. It is not hard to see that
\[
\binom{hp^r}{p^r}\equiv h \neq 0 ~ (\text{modulo $p$}).
\]
Hence since $\phi(x_i)^d\in I$,
\[
\binom{d}{p^r}(cx_i)^{p^r}(c_jx_j)^{d-p^r} \in I,
\]
so ${x_i}^{p^r}{x_j}^{d-p^r} \in I$. This is a contradiction.

Thus we have $\phi(x_i)=cx_i + \sum_{j=1}^tc_jx_j$. The right-hand side is in $R$, so apply $\phi$ again we have
\[
cx_i + \sum_{j=1}^tc_jx_j=c(cx_i + \sum_{j=1}^tc_jx_j)+\sum_{j=1}^tc_jy_j.
\]
Since $y_j\in (x_1,\ldots,x_t)$ for $j=1,\ldots,t$, we have $cx_i=c^2x_i$. Hence $c=1$, and thus
$\phi(x_i)=x_i+\sum_{j=1}^{t}c_jx_j$. Therefore (i) holds and the claim is true.

{\em Case 2}. In this case, $\chara k=p>0$, $d=p^r$ where $r\ge 1$. Choose a set $V\subseteq \{t+1,\ldots,n\}$ such that
$\{y_i:i\in U\cup V\}$ is a $k$-basis for $R_1$. We will show that $W=U\cup V$ satisfies two conditions of a base. The first
is obvious from the choice of $U$: $\{y_i:i\in W\}$ are linearly independent over $k$. For the second condition, let $i \in \{t+1,\ldots,n\}\setminus W$.  Then
\[
y_i=\sum_{j\in W}c_jy_j, ~\text{where $c_j\in k$ for each $j\in W$}.
\]
This implies that $y_i^d =\sum_{j\in W}c_j^dy_j^d \in I$. Thus $\phi(I)=\phi(I_WS)$. The proof of the theorem is now complete.
\end{proof}

The following result also supports Conjecture \ref{retracts_of_monomial_quotient_rings}.

\begin{prop}
\label{power_of_linear_ideals}
If $I=(x_{i_1},\ldots,x_{i_t})^d$ where $d\ge 2$, then Conjecture \ref{retracts_of_monomial_quotient_rings} is true.
\end{prop}
\begin{proof}
Of course it is harmless to assume that $i_1=1,\ldots,i_t=t$. Similar to the proof of Theorem \ref{irreducible_monomial_quotients}, we can reduce
to the case $I=(x_1,\ldots,x_n)^d$.

Assume that $R/J \hookrightarrow S/I$ is a $\Z$-graded algebra retract with the retraction map $S/I \to R/J$, where $R$ is a standard graded
polynomial ring, $R/J$ is in minimal presentation form. Let $R\hookrightarrow S$ be the induced algebra retract and $\phi: S\to R$ the retraction
map.

Choose $W\subseteq [n]$ such that $\{\phi(x_i):i\in W \}$ is a $k$-basis of $R_1$. It is easy to see that
$\phi(I)=(\phi(x_i):i\in W)^d.$ This concludes the proof of the proposition.
\end{proof}

\section*{Acknowledgments}
We are grateful to Aldo Conca, Winfried Bruns and Tim R\"omer for their helpful comments and suggestions. We would like to thank Holger Brenner for his critical comments on the writing.

We would like to thank the referee for his/her thorough reading and many thoughtful suggestions that helped us to improve the exposition.

\end{document}